\documentclass[11pt]{article}

\usepackage[english]{babel}
\usepackage[a4paper,top=2cm,bottom=3cm,left=2cm,right=2cm,marginparwidth=1.75cm]{geometry}
\usepackage[utf8]{inputenc}

\usepackage[colorlinks=true, allcolors=blue]{hyperref}
\usepackage{amsmath, amsthm, amssymb, mathtools}
\usepackage{graphicx}
\usepackage{todonotes}
\usepackage{cite}

\usepackage{xcolor, colortbl}
\definecolor{Gray}{gray}{0.85}

\def\Q{\mathbb{Q}}
\def\Z{\mathbb{Z}}
\def\N{\mathbb{N}}
\def\R{\mathbb{R}}
\def\C{\mathbb{C}}

\def\B{\mathcal{B}}
\def\BB{\boldsymbol{\B}}

\DeclareMathOperator{\Fin}{Fin}

\newcommand{\floor}[1]{\lfloor #1 \rfloor}

\makeatletter
\newcommand*\bigcdot{\mathpalette\bigcdot@{.5}}
\newcommand*\bigcdot@[2]{\mathbin{\vcenter{\hbox{\scalebox{#2}{$\m@th#1\bullet$}}}}}
\makeatother

\theoremstyle{definition}
\newtheorem{definition}{Definition}

\newtheorem{remark}[definition]{Remark}
\newtheorem{example}[definition]{Example}

\theoremstyle{plain}
\newtheorem{theorem}[definition]{Theorem}
\newtheorem{proposition}[definition]{Proposition}
\newtheorem{lemma}[definition]{Lemma}

\newtheorem{theoremcisla}{Theorem}

\title{Periodicity and pure periodicity in alternate base systems}
\author{Zuzana Masáková and Edita Pelantová\\[2mm] 
Department of Mathematics FNSPE\\
Czech Technical University in Prague\\
zuzana.masakova@fjfi.cvut.cz,
 edita.pelantova@fjfi.cvut.cz
}

\begin{document}
\maketitle

\begin{abstract}
We study the Cantor real base numeration system which is a common generalization of two positional systems, namely the Cantor system with a sequence of integer bases and  the Rényi system with one real base.
We focus on the so-called alternate base 
$\BB$  given by a purely periodic sequence of real numbers greater than 1.
We answer an open question of Charlier et al.\ on the set of numbers with eventually periodic $\BB$-expansions.
We also investigate for which bases  all sufficiently small rationals have a purely periodic $\BB$-expansion. 
\end{abstract}

\noindent


\section{Introduction}\label{sec:intro}


Cantor real base systems were first studied by Caalim and Demglio in~\cite{CD20} and independently by Charlier and Cisternino in~\cite{CC21} as a generalisation of Rényi $\beta$-expansions~\cite{Renyi57}. While in the Rényi numeration system, one uses for representation of numbers a sum of powers of a single base $\beta>1$, here we 
consider a sequence of real bases $\BB=(\beta_i)_{i\geq 1}$, $\beta_i>1$. A real number $x\in[0,1)$ can be represented by an infinite series 
$$
x=\sum_{k=1}^{+\infty} \frac{x_k}{\prod_{i=1}^k \beta_i}, \qquad x_k\in\N.
$$
Note that the possibility to represent real numbers in this form was already mentioned in~\cite{Galambos}. Some of the properties of such representations are direct analogies of those proved for $\beta$-expansions, others appear to be much more difficult.

Papers~\cite{CD20} and~\cite{CC21} concentrated on characterizing the representations which are produced by the greedy algorithm, the so-called $\BB$-expansions. The characterization is given in terms of a set of lexicographic conditions, which are to be compared  to those obtained by Parry~\cite{P60} for Rényi $\beta$-expansions.
Charlier and Cisternino~\cite{CC21} then focused on the sequences of bases that are purely periodic with period of length $p$. They called such a base $\BB$ an alternate base, and write $\BB=(\beta_1,\dots,\beta_p)$. They then characterized alternate bases providing sofic systems. Algebraic description of sofic alternate bases is given in~\cite{CCMP23}. Note that for $p=1$, one obtains the case of Rényi $\beta$-expansions where soficness was described by Bertrand-Mathis~\cite{Bertrand-Mathis}.

From the arithmetical point of view, one is interested which numbers have $\BB$-expansions with finite, purely periodic or eventually periodic $\BB$-expansions.
The so-called finiteness property (F), i.e. the fact that addition and subtraction of finite $\BB$-expansions yields again a finite $\BB$-expansions, was studied in~\cite{MPS23-LNCS}, providing some necessary and some sufficient conditions for finiteness, a counterpart of the results of Frougny and Solomyak~\cite{FS92} and others. A class of bases with (F) property was also given.

The purpose of this article is to study the set ${\rm Per}(\BB)$ of numbers in the unit interval $[0,1)$ with periodic expansions in alternate base numeration systems. For $p=1$, Schmidt~\cite{S80} has shown that if ${\rm Per}(\beta)$ contains all rational numbers of $[0,1)$, then $\beta$ is a Pisot number or a Salem number. 
As a partial converse, Schmidt~\cite{S80} has also proved that if $\beta$ is a Pisot number, then ${\rm Per}(\beta)=\Q(\beta)\cap[0,1)$, where $\Q(\beta)$ denotes the algebraic extension of rational numbers by $\beta$.

Charlier et al.~\cite{CCK23} have shown the analogy of the above result of Schmidt for alternate base numeration systems.

\begin{theorem}[\cite{CCK23}]\label{thm:PKreczman}
Let $\BB=(\beta_1,\dots,\beta_{p})$ be an alternate base and set 
$\delta=\prod_{i=1}^p \beta_i$.
\begin{itemize}
\item[(1)] If $\Q\cap[0,1) \subseteq \bigcap_{i=1}^p {\rm Per}(\BB^{(i)})$,
then $\delta$ is either a Pisot  or a Salem number and $\beta_1, \dots, \beta_{p} \in \Q(\delta)$.
\item[(2)] If $\delta$ is a Pisot number and  $\beta_1, \dots, \beta_{p} \in \Q(\delta)$, then ${\rm Per}(\BB) = \Q(\delta) \cap [0,1)$.
\end{itemize}
\end{theorem}

Note that for the implication (1), the authors require that rational numbers have eventually periodic expansion not only in the alternate base $\BB=(\beta_1,\dots,\beta_p)$, but also in all its shifts, $\BB^{(i)}=(\beta_i,\beta_{i+1}, \dots,\beta_{i+p})$, $i=1,\dots,p$, where the indices are counted modulo $p$. The authors state as a question, whether one can conclude the same requiring only 
$\Q\cap[0,1) \subseteq {\rm Per}(\BB)$. In this paper we answer this question in the affirmative. In Section~\ref{sec:ThmA}, we prove with the help of several auxiliary statements the following theorem. 

\begin{theoremcisla}\label{thm:SchmidtLepsi}
    Let $\BB=(\beta_1,\dots,\beta_{p})$ be an alternate base and set 
$\delta=\prod_{i=1}^p \beta_i$. If $\Q\cap[0,1) \subseteq {\rm Per}(\BB)$,
then $\delta$ is either a Pisot  or a Salem number and $\beta_1, \dots, \beta_{p} \in \Q(\delta)$.
\end{theoremcisla}

The second part of our results concerns rational numbers with purely periodic $\BB$-expansions. We will say that an alternate base $\BB$ satisfies pure periodicity property (Property (PP)), if there exists an interval $[0,\gamma)$, $0<\gamma\leq1$, such that every rational in $[0,\gamma)$ has purely periodic $\BB$-expansion. A non-trivial problem is determination of the supremum of all constants $\gamma$ exhibiting Property (PP) in base $\B$. Let us denote it  by $\gamma(\BB)$.

Before stating our results, let us recall what is known for the case when $p=1$. For Rényi $\beta$-expansions, Schmidt has shown that quadratic Pisot units with minimal polynomial $x^2-mx-1$, $m\geq 1$, satisfy (PP), moreover with $\gamma(\beta)=1$. Later, Hama and Imahashi~\cite{HI97} derived that if $\beta$ is a quadratic Pisot unit not of this type (i.e.\ with minimal polynomial $x^2-mx+1$, $m\geq 3$), then no rational number has purely periodic $\beta$-expansion, thus $\beta$ does not possess (PP).

Akiyama~\cite{A98} has put Property (PP) into connection with the finiteness property. In particular, he proved the following. 

\begin{theorem}[\cite{A98}]
Let $\beta>1$ satisfies (PP), then $\beta$ is a Pisot unit. On the other hand, if $\beta>1$ is a Pisot unit with Property (F), then $\beta$ satisfies (PP).
\end{theorem}

The question whether validity of (F) is necessary for (PP) has been decided for quadratic bases $\beta$ (as a result of Schmidt~\cite{S80} and Hama and Imahashi~\cite{HI97}) and also for cubic bases $\beta$. This is a result of Adamczewski et al.~\cite{AFSS10} who prove that a cubic base $\beta$ satisfies (PP) if and only if it is a Pisot unit with (F). Moreover, they show that the constant $\gamma(\beta)$ from Property (PP)  is irrational for cubic numbers which are not totally real.

In this paper we study Property (PP) of alternate bases. We show a necessary condition.

\begin{theoremcisla}\label{thm:necessary}
Let $\BB=(\beta_1,\dots,\beta_p)$ be an alternate base with Property (PP). Then $\delta=\prod_{i=1}^p\beta_i$ is a Pisot or a Salem unit and $\beta_i\in\Q(\delta)$ for every $i=1,\dots,p$. Moreover, the vector $(\psi(\beta_1),\dots,\psi(\beta_p))$ is not positive for any non-identical embedding $\psi:\Q(\delta)\to\C$.
\end{theoremcisla}

The proof of Theorem~\ref{thm:necessary}, is provided in Section~\ref{sec:necessary}.
A partial converse of Theorem~\ref{thm:necessary}, a sufficient condition for (PP), is the following.

\begin{theoremcisla}\label{thm:sufficient}
 Let $\BB=(\beta_1,\dots,\beta_p)$ be an alternate base with Property (F) such that $\delta=\prod_{i=1}^p\beta_i$ is a Pisot unit. Then $\BB^{(i)}$ satisfies (PP) for every $i=1,\dots,p$.
\end{theoremcisla}

Theorem~\ref{thm:sufficient} is shown in Section~\ref{sec:sufficient}. 
In the last Section~\ref{sec:example} we give a class of alternate bases with (PP) for which the constant $\gamma(\BB)$ is equal to 1. We also illustrate the fact that $\gamma(\BB^{(i)})$ may be different from $\gamma(\BB)$.  

%

%

\section{Preliminaries}\label{sec:preli}

Cantor real base is given by a sequence $\BB=(\beta_k)_{k\geq 1}$
of real numbers $\beta_k>1$. Any $x\in[0,1]$ is represented in $\BB$ as a series
$$
x=\sum_{k=1}^{+\infty} \frac{x_k}{\prod_{i=1}^k \beta_i} \qquad \text{with } x_k\in\N.
$$
The sequence of integer digits $x_1x_2x_3\cdots$ is called a $\BB$-representation of $x$. The greatest $\BB$-representation of $x$ in lexicographic order, called the $\BB$-expansion of $x$, is the one obtain by the greedy algorithm:
Set $r_0=x$, and for $k\geq 0$ set $a_{k+1}=\floor{\beta_{k+1} r_{k}}$, $r_{k+1}=\beta_{k+1} r_{k} - a_{k+1}$. We denote the $\BB$-expansion of $x$ by $d_{\BB}(x)=a_1a_2a_3\cdots$. Note that $0\leq a_k < \beta_k$. 
Setting $r_0=1$, one defines the greedy expansion of 1, $d_{\BB}(1)=t_1t_2t_3\cdots$. 

For characterisation of integer sequences that are admissible as $\BB$-expansions of numbers from the interval $[0,1)$, one needs to define the quasigreedy expansion of 1, denoted $d_{\BB}^*(1)$, as the lexicographically greatest $\BB$-representation of 1 with infinitely may non-zero digits. 

Given a Cantor real base $\BB=(\beta_1,\beta_2,\beta_3,\dots)$, denote for $i\geq 1$ the shift  of the base, $\BB^{(i)}=(\beta_i,\beta_{i+1},\beta_{i+2},\dots)$. The following statement was proved in~\cite{CC21}.

\begin{theorem}[\cite{CC21}]\label{thm:admis}
Let $\BB=(\beta_k)_{k\geq 1}$ be a Cantor real base. A sequence of integers $x_1x_2x_3\cdots$ is a $\BB$-expansion of a number $x\in[0,1)$ if and only if for every $i\geq 1$ we have
$$
0^\omega \preceq x_ix_{i+1}x_{i+2}\cdots \prec d_{\BB^{(i)}}^*(1).
$$
\end{theorem}
By $\preceq$ we denote the standard lexicographic order; $w^\omega$ stands for infinite repetition of the string $w$.

If the base $\BB$ is a purely periodic sequence with period length $p$, i.e.\ $\beta_{k+p}=\beta_k$ for any $k\geq 1$, then $\BB^{(k+p)}=\BB^{(k)}$ for any $k\geq 1$. In this case we speak about an alternate base and write $\BB=(\beta_1,\dots,\beta_p)$. The special case when $p=1$ corresponds to the numeration system with a single base $\beta>1$, as was defined by Rényi and extensively studied by many authors from very diverse points of view.

In~\cite{MPS23-LNCS} the set ${\rm Fin}(\BB)$ of numbers with $\BB$-expansions having only finitely many non-zero digits is considered. We call such expansions finite. We denote
$$
{\rm Fin}(\BB)=\{x\in[0,1):  d_{\BB}(x)\text{ is finite}\}.
$$
We say that the base $\BB$ satisfies the finiteness property (F), if 
for any $x,y\in{\rm Fin}(\BB)$, we have
\begin{equation}\label{eq:F} 
x+y\in[0,1)\implies x+y\in\Fin(\BB)\quad \text{and}\quad 
x-y\in[0,1)\implies x-y\in\Fin(\BB).
\end{equation}
In~\cite{MPS23-LNCS}, some necessary and some sufficient conditions for an alternate base $\BB$ with period $p$ to satisfy (F) are presented. Among other, it is shown that if $\BB$ satisfies the finiteness property, then $\delta=\prod_{i=1}^p\beta_i$ is a Pisot or a Salem number, $\beta_i\in\Q(\delta)$ and for any non-identical embedding $\psi$ of $\Q(\delta)$ into $\C$, the vector $(\psi(\beta_1),\dots,\psi(\beta_p))$ is not positive.

Recall that a complex number $\delta>1$ is a Pisot number, if it is an algebraic integer, i.e.\ a root of a monic polynomial with integer coefficients, with all conjugates in the interior of the unit circle. The number $\delta>1$ is a Salem number, if it is an algebraic integer with all conjugates in the unit circle and at least one of modulus equal to 1. The algebraic extension of rational numbers by $\delta$ is denoted by $\Q(\delta)=\{a_0+a_1\delta+\cdots a_{n-1}\delta^{n-1}:a_i\in\Q\}$, where $n$ is the degree of $\delta$ as an algebraic number. Such a field $\Q(\delta)$ has $n$ embeddings into $\C$ (including the identity), i.e. field monomorphisms $\psi: \Q(\delta)\to\C$, induced by $\delta\mapsto \delta'$ where $\delta'$ is a conjugate of $\delta$.

In this paper we are particularly interested in numbers with eventually and purely periodic $\BB$-expansions. According to~\cite{CCK23}, we define 
$$
{\rm Per}(\BB)=\{x\in[0,1): d_{\BB}(x)\text{ is eventually periodic} \}.
$$

The second part of this paper is focused to rational numbers with purely periodic $\BB$-expansion. 

\begin{definition}\label{def:PuPe} An alternate base $\BB = (\beta_1, \beta_2, \ldots, \beta_p)$  has the Pure Periodicity Property (PP), if there exists $\gamma > 0$ such that $  d_{\BB}(x)$  is purely periodic for every  $x \in [0, \gamma)\cap \Q$. 
\end{definition}

In the proof of our result given in Theorem~\ref{thm:sufficient}, we will need to extend the definition of a $\BB$-expansion to numbers outside of the unit interval. In~\cite{CCMP24}, this is done with full generality for a two-way Cantor real base, here we simplify the task by considering an alternate base $\BB=(\beta_1,\dots,\beta_p)$, with $\delta=\prod_{i=1}^p\beta_i$.
For a given non-negative number $x$ find $k\in\N$ such that 
$z=\frac{x}{\delta^k}\in[0,1)$. Denote $d_{\BB}(z)=a_1a_2a_3\cdots$. Then set 
$$
d_{\BB}(x)=a_1a_2\cdots a_{pk}\bigcdot a_{pk+1}a_{pk+2}\cdots.
$$
Having as a convention that two expansions coincide if they are the same up to leading zeros, $d_{\BB}(x)$ is unique not dependent on the choice of $k$. 
With this in hand, we can define the set 
$$
{\rm fin}(\BB)=\pm\bigcup_{k\in\N} \delta^k {\rm Fin}(\BB).
$$
which gives the set of all real numbers whose absolute value has a finite $\BB$-expansion. Property (F) then translates to saying that ${\rm fin}(\BB)$ is closed under addition and subtraction.

We further define the $\BB$-integers as numbers having only zeros on the right from the fractional point. We denote
$$
\N_{\BB} =  \{x \geq 0: d_{\BB}(x) = a_1\cdots a_n \bigcdot 0^\omega\}.
$$


\section{Proof of Theorem~\ref{thm:SchmidtLepsi}}\label{sec:ThmA}

The expansion of a real number $x\in[0,1)$ in the alternate base $\BB=(\beta_1,\dots,\beta_p)$ is of the form
$d_{\BB}(x)=x_1x_2x_3\cdots$ with integer digits $x_k\in\N$, $x_k < \beta_k$. We can rewrite 
\begin{equation}\label{eq:uprava}
x=\sum_{k=1}^{+\infty} \frac{x_k}{\prod_{i=1}^k \beta_i}  = \sum_{k=1}^{+\infty}\delta^{-k}\  \sum_{j=1}^p\ x_{(k-1)p+j}\, \Bigl(\!\!\prod_{i=j+1}^p\!\! \beta_i\Bigr)  = \sum_{k=1}^{+\infty}  d_k\, \delta^{-k}\ .
\end{equation}
The latter can be viewed as a representation of $x$ in the base $\delta=\prod_{i=1}^p\beta_i$ with digits 
$d_k$ belonging to the alphabet 
\begin{equation}\label{eq:Dabeceda}
\mathcal{D} := \Bigl\{a_{1} \Bigl(\,\prod_{i=2}^p\beta_i\Bigr)  + a_2 \Bigl(\,\prod_{i=3}^p\beta_i\Bigr)+ \cdots + a_{p-1}\beta_p + a_p: a_k \in \mathbb{N}, a_k< \beta_k\Bigr\}.
\end{equation}

In order to simplify the notation, denote 
\begin{equation}\label{eq:v}
\vec{v} =(v_1,\dots,v_p)^T= \! \bigg(\!\prod\limits_{i=2}^p \beta_i, \prod\limits_{i=3}^p \beta_i,
\ldots,  \prod\limits_{i=p}^p \beta_i, 1\bigg)^T .    
\end{equation}
Then we can express the alphabet $\mathcal{D}$ as
\begin{equation}\label{eq:Dabeceda1}
 \mathcal{D} = \big\{(a_1,a_2, \ldots, a_p) \vec{v} :  a_k \in \mathbb{N}, a_k< \beta_k \big\}.   
\end{equation}

Suppose now that the $\BB$-expansion of $x$ is eventually periodic. The lengths of the preperiod and the period can always be assumed to be multiples of $p$, say
$$
d_{\BB}(x) = x_1x_2\cdots x_{pr}(x_{pr+1}x_{pr+2}\cdots x_{p(r+s)})^\omega,
$$
which yields an eventually periodic $\delta$-representation of $x$, say 
$d_1\cdots d_r(d_{r+1}\cdots d_{r+s})^\omega$.
For the value of $x$ we can therefore write
$$
\begin{aligned}
x&=\frac{d_1}{\delta}+\cdots+\frac{d_r}{\delta^r}+\big(\frac{d_{r+1}}{\delta^{r+1}}+\cdots+\frac{d_{r+s}}{\delta^{r+s}}\big)\sum_{i=0}^{\infty}\delta^{is}=\\
&=\frac{1}{\delta^r}(d_1\delta^{r-1}+d_2\delta^{r-2}+\cdots +d_r) + \\
&\qquad +
\frac{1}{\delta^{r}(\delta^s-1)}(d_{r+1}\delta^{s-1}+d_{r+2}\delta^{s-2}+\cdots+d_{r+s})\,,
\end{aligned}
$$
which gives
$$
x\delta^r(\delta^s-1) = (\delta^s-1)\sum_{k=1}^r d_k\delta^{r-k} + \sum_{k=1}^sd_{r+k}\delta^{s-k}\,.
$$

Realizing that the digits $d_k\in\mathcal{D}$ are of the form
$$
d_k = \bigl(x_{p(k-1)+1}, x_{p(k-1)+2}, \ldots, x_{p(k-1)+p-1}, x_{pk} \bigr)\vec{v},
$$ 
we can rewrite the value of the product $x\delta^r(\delta^s-1)$ as
\begin{equation}\label{eq:expevper}
x\delta^r(\delta^s-1) = \Big( (\delta^s-1)g_1(\delta)+f_1(\delta),\dots,(\delta^s-1)g_p(\delta)+f_p(\delta)\Big)\vec{v}\,,
\end{equation}
where $g_i$, $f_i$, $i=1,\dots,p$, are polynomials with integer non-negative coefficients 
$$
g_i(X) = \sum_{k=1}^r x_{p(k-1)+i}X^{r-k}\,,\qquad
f_i(X) = \sum_{k=r+1}^{r+s} x_{p(k-1)+i}X^{r+s-k}\,,
$$ 
of degrees $\deg g_i \leq r-1$, $\deg f_i \leq s-1$.

In the particular case where the $\BB$-expansion of $x$ is purely periodic, we have $r=0$, the polynomials $g_i$ for $i=1,\dots,p$ vanish and we can simplify to
\begin{equation}\label{eq:Xperiodic}
\begin{gathered}  
x(\delta^s-1)=\sum_{k=1}^sd_{k}\delta^{s-k}=\big( f_1(\delta),\dots,f_p(\delta)\big)\vec{v}\,,\\
\quad\text{with}\quad 
f_i(X) = \sum_{k=1}^{s} x_{p(k-1)+i}X^{s-k}\,. 
\end{gathered}
\end{equation}

Now assume that we have $p$ rational numbers $\frac{p_j}{q_j}$, $j=1,\dots,p$, with eventually periodic $\BB$-expansions $d_{\BB}(\frac{p_j}{q_j})=x_1^{(j)}x_2^{(j)}\cdots$.
Without loss of generality, we can assume that all the expansions have the preperiod $pr$ and the period $ps$ of the same length, i.e.
\begin{equation}\label{eq:poslperrozvojzlomky}
d_{\BB}\big(\tfrac{p_j}{q_j}\big)=x_1^{(j)}x_2^{(j)}\cdots x_{pr}^{(j)}\big(x_{pr+1}^{(j)}x_{pr+2}^{(j)}\cdots x_{p(r+s)}^{(j)}\big)^\omega.
\end{equation}
Then we have $p$ equalities of the form~\eqref{eq:expevper} that can be rewritten together into a matrix form, $M\vec{v}=\vec{0}$, where the matrix $M$ is given by
\begin{equation}\label{eq:matice M}
M= \left(
\begin{array}{ccccc}
q_1 h^{(1)}_1(\delta)& q_1 h^{(1)}_2(\delta)& \cdots & q_1 h^{(1)}_{p-1}(\delta)& q_1 h^{(1)}_{p}(\delta) - p_1\delta^r(\delta^s\!-\!1)\\[2mm]
q_2 h^{(2)}_1(\delta)& q_2 h^{(2)}_2(\delta)& \cdots & q_2 h^{(2)}_{p-1}(\delta)& q_2 h^{(2)}_{p}(\delta) -p_2\delta^r(\delta^s\!-\!1)\\
\vdots& \vdots & & \vdots & \vdots\\[2mm]
q_p h^{(p)}_1(\delta)& q_p h^{(p)}_2(\delta)& \cdots & q_p h^{(p)}_{p-1}(\delta)& q_p h^{(p)}_{p}(\delta) -p_p\delta^r(\delta^s\!-\!1)\\
\end{array}\right)\,,   
\end{equation}
where for simplicity we have denoted
$$
h_i^{(j)}(X) = (X^s-1)g_i^{(j)}(X) + f_i^{(j)}(X)
$$
and
\begin{equation}\label{eq:fg}
g_i^{(j)}(X) = \sum_{k=1}^r x_{p(k-1)+i}^{(j)}X^{r-k}\,,\qquad
f_i^{(j)}(X) = \sum_{k=r+1}^{r+s} x_{p(k-1)+i}^{(j)}X^{r+s-k}\,.    
\end{equation}


\begin{lemma}\label{l:omatici}
Let $\BB=(\beta_1,\dots,\beta_p)$ be an alternate base and $\delta=\prod_{i=1}^p\beta_i$. Suppose there exists a non-singular $p\times p$ matrix $M(X)$ whose entries are integer polynomials, i.e.\ belong to $\Z[X]$. Let $M(\delta)\vec{v}=\vec{0}$, where $\vec{v}$ is given by~\eqref{eq:v}. Then the following hold.
\begin{enumerate}
    \item[1)] $\delta$ is an algebraic number.
    \item[2)] If the rank of $M(\delta)$ is $p-1$, then $\beta_i\in\Q(\delta)$ for every $i=1,2,\dots,p$.
\end{enumerate}
\end{lemma}

\begin{proof}
 By assumption, the determinant of $M(X)$ is a non-zero polynomial with integer coefficients, say $\det M(X) = F(X)$. For Item 1), it suffices to realize that $\vec{v}$ is non-zero, and thus the matrix $M(\delta)$ must be singular. We have $\det M(\delta)=0 = F(\delta)$, which proves that $\delta$ is an algebraic number.

 Let us prove Item 2). Since  $M(\delta)\vec{v}=\vec{0}$, the vector $\vec{v}$ is an eigenvector of $M(\delta)$ corresponding to the eigenvalue 0. As rank of $M(\delta)$ is $p-1$, the corresponding eigenvector is unique up to multiplication by a constant. 
 In particular, for any real vector $\vec{u}$ satisfying $M(\delta)\vec{u}=\vec{0}$ there exists $\alpha\in\R$ such that $\vec{u}=\alpha\vec{v}$.

 Since the entries of $M(\delta)$ belong to $\Q(\delta)$, we can choose the eigenvector $\vec{u}$ to have entries in $\Q(\delta)$. We then have
 $$
 \beta_i = \frac{v_{i-1}}{v_i} = \frac{u_{i-1}}{u_i}\in\Q(\delta)\,,\ \text{for } i=2,3,\dots,p,\ \text{ and }\  \beta_1=\frac{\delta}{\beta_2\cdots\beta_p}\in\Q(\delta).
 $$
\end{proof}

In what follows, we will set the choice of the rational numbers $\frac{p_j}{q_j}$ so that the matrix $M$ of~\eqref{eq:matice M} satisfies the assumptions of Lemma~\ref{l:omatici}.

\begin{lemma}\label{l:deltaAlgebraic}  
Let $\BB = (\beta_1, \beta_2, \ldots, \beta_p)$ be an alternate base.
Suppose that there exists a constant $\gamma>0$ such that every rational number in $[0,\gamma)$ has eventually periodic $\BB$-expansion. Then 
\begin{enumerate}
\item[1)]
$\delta=\prod_{i=1}^p\beta_i$ is an algebraic number, 
\item[2)] 
$\beta_i\in\Q(\delta)$ for $i=1,\dots,p$, and 
\end{enumerate}
\end{lemma}

\begin{proof}  
We will first make a suitable choice of $p$ rational numbers  $x^{(j)}=\frac{p_j}{q_j}, j=1,2,\ldots, p$  from the interval  $[0,\gamma)$ and form a matrix $M$ of the form~\eqref{eq:matice M}. For the proof of the statement, we than use Lemma~\ref{l:omatici}. 

Fix positive  $m\in \N$ and $n \in \N$ such that $\tfrac{1}{\delta^n} < \gamma$. For every $j=1,2,\ldots,p$ denote 
\begin{equation}\label{eq: Iintrvaly}
I_j= \{x \in [0,1): d_{\BB}(x) \ \text{has a prefix \ } 0^{j-1}10^{pm-j}\}.
\end{equation}
Since the ordering of numbers in the interval $(0,1)$
corresponds to lexicographic order of their $\BB$-expansions, the intervals $I_j$  are mutually disjoint.
For each $j =1,2,\ldots, p$ choose a rational number $x^{(j)} = \frac{p^j}{q_j}$ from the interval $\tfrac{1}{\delta^n}I_j$. 
The  $\BB$-expansion of $x$ is then of the form
\begin{equation}\label{eq:spectvar}
d_{\BB}(x^{(j)}) = 0^{pn+j-1}10^{pm-j} x^{(j)}_{p(n+m)+1}x^{(j)}_{p(n+m)+2}\cdots.
\end{equation}
Thanks to the choice of $n$, each interval  $\tfrac{1}{\delta^n}I_j $ is a subset of $\subset [0, \gamma)$, and therefore $d_{\BB}(x^{(j)})$ is eventually periodic. 
Without loss of generality we assume that the preperiod of $d_{\BB}(x^{(j)})$ is of length $pr>p(n+m)$ and the period is of length $ps>0$. Thus $d_{\BB}(x^{(j)})$ is of the form~\eqref{eq:poslperrozvojzlomky} and we have a matrix equation $M\vec{v}=\vec{0}$ for a matrix $M$ as in~\eqref{eq:matice M}, where for $i,j=1,\dots,p$ we have
$$
h_i^{(j)}(X) = (X^s-1)g_i^{(j)}(X) + f_i^{(j)}(X),
$$
and the polynomials $f_i^{(j)}$ and $g_i^{(j)}$ are as in~\eqref{eq:fg}. 

In order to show that $\delta$ is an algebraic number, by Lemma~\ref{l:omatici}, it suffices to verify that the determinant of $M$ is equal to $\det M = F(\delta)$ for some non-zero polynomial $F \in \Z[X]$.
For that, it suffices to show that in each row of the matrix $M(X)$, the degree of the polynomial at the diagonal is strictly larger than the degrees of polynomials at other positions in the row.

Since for every $j=1,\dots,p$ the digits $x_k^{(j)}$ of $d_{\BB}(x^{(j)})$ satisfy
$0\leq x_{k}^{(j)}< \beta_k$, and moreover
$$
x_{k}^{(j)}=
\begin{cases}
0 &\text{ for } 1\leq k\leq pn+j-1,\\
1 &\text{ for } k= pn+j,\\
0 &\text{ for } pn+j+1\leq k\leq p(n+m).
\end{cases},
$$
we can derive that 
\begin{equation}\label{eq:g}
  g_{i}^{(j)}(X) = \sum_{k=1}^r x_{p(k-1)+i}^{(j)}X^{r-k} = 
\begin{cases}
\displaystyle{\sum_{k=n+m+1}^r x_{p(k-1)+i}^{(j)}X^{r-k}} , &\text{for } i\neq j,\\[2mm]
\displaystyle{X^{r-n-1} + \sum_{k=n+m+1}^r x_{p(k-1)+i}^{(j)}X^{r-k}}  &\text{for } i= j.
\end{cases}
\end{equation}
For the degrees of the polynomials $h_i^{(j)}$ we therefore have
$$
\begin{aligned}
{\rm deg}\,h_i^{(j)} &\leq s+r-n-m-1, &\text{ for } i\neq j,\\ 
{\rm deg}\,h_i^{(i)} &= s+r-n-1.
\end{aligned}
$$
Moreover, the polynomial $h_i^{(i)}$ is monic for $i=1,\dots,p$, and its degree is strictly larger than the degree of polynomials $h_i^{(j)}$, $i\neq j$.

Formula for computation of the determinant of the matrix $M$ ensures that $\det M = F$ has the same leading coefficient as the product of polynomials on the diagonal of $M(\delta)$,
$$
q_1h^{(1)}_1(X)\, q_2h^{(2)}_2(X)\, \cdots  q_{p-1}h^{(p-1)}_{p-1}(X)\, \bigl(q_ph^{(p)}_p(X) - p_pX^r(X^s -1)\bigr). 
$$ 
Thus $\deg F = (s+r-n-1)(p-1)+s+r$ and the leading coefficient of $F$ is $-q_1q_2 \ldots q_{p-1}p_p$. The number $\delta$ is therefore a root of a non-zero polynomial  $F\in \mathbb{Z}[X]$ and hence $\delta$ is an algebraic number. 

\bigskip
Let us now demonstrate Item 2) of the statement. By Lemma~\ref{l:omatici}, it suffices to show that $M(\delta)$ is equal to $p-1$. We will ensure this fact by choosing a suitable parameter $m \in \N$. Let us stress that so-far our considerations used arbitrary positive integer $m\in \N$.  

We will prove that  for sufficiently large $m$ the submatrix 
\begin{equation}\label{eq:submatice}
\left(
\begin{array}{cccc}
q_1 h^{(1)}_1(\delta)& q_1 h^{(1)}_2(\delta)& \cdots & q_1 h^{(1)}_{p-1}(\delta)\\[2mm]
q_2 h^{(2)}_1(\delta)& q_2 h^{(2)}_2(\delta)& \cdots & q_2 h^{(2)}_{p-1}(\delta)\\
\vdots& \vdots & & \vdots \\[2mm]
q_{p-1} h^{(p-1)}_1(\delta)& q_{p-1} h^{(p-1)}_2(\delta)& \cdots & q_{p-1} h^{(p-1)}_{p-1}(\delta)\\
\end{array}\right)   
\end{equation}
of the matrix $M$ is strictly diagonally dominant and thus it is non-singular. 
For the $j$-th row, we need to verify that 
\begin{equation}\label{eq:diagdom}
|h_{j}^{(j)}(\delta)| > \Big| \sum_{i=1,i\neq j}^{p-1} 
h_i^{(j)}(\delta)\Big|.
\end{equation}
From~\eqref{eq:fg} and~\eqref{eq:g}, and the fact that coefficients of all polynomials are non-negative, we can deduce the following estimates on $f_i^{(j)}(\delta)$, $g_i^{(j)}(\delta)$.
We have 
$$
\begin{aligned}
0&\leq g_{i}^{(j)}(\delta)< 
\floor{\beta_i} \frac{\delta^{r-n-m}}{\delta-1} &\text{ for } i\neq j,\\
\delta^{r-n-1}&\leq g_{i}^{(j)}(\delta) <
\delta^{r-n-1}+\floor{\beta_i} \frac{\delta^{r-n-m}}{\delta-1}, &\text{ for } i= j,\\
0&\leq f_i^{(j)}(\delta) \leq \floor{\beta_i} \frac{\delta^s-1}{\delta-1}.    
\end{aligned}
$$ 
Thus $h_i^{(j)}(X) = (X^s-1)g_i^{(j)}(X) + f_i^{(j)}(X)$ satisfy
\begin{equation}\label{eq:diag}
|h_{j}^{(j)}(\delta)| \geq \delta^{r-n-1}(\delta^s-1)
\end{equation}
and
\begin{equation}\label{eq:mimodiag}
\begin{aligned}
\Big| \sum_{i=1,i\neq j}^{p-1} 
h_i^{(j)}(\delta)\Big| &< \frac{\delta^s-1}{\delta-1} 
(\delta^{r-n-m}+1)
\sum_{i=1,i\neq j}^{p-1} \floor{\beta_i} \\
&< 2\delta^{r-n-m}\frac{\delta^s-1}{\delta-1} 
\sum_{i=1,i\neq j}^{p-1} \floor{\beta_i}.
\end{aligned}
\end{equation}
It is therefore sufficient to show that the right hand-sides of~\eqref{eq:diag} and~\eqref{eq:mimodiag} satisfy
$$
\delta^{r-n-1}(\delta^s-1) > 2\delta^{r-n-m}\frac{\delta^s-1}{\delta-1} 
\sum_{i=1,i\neq j}^{p-1} \floor{\beta_i},
$$
or equivalently 
$$
\delta^{m-1}(\delta-1) > 2 
\sum_{i=1,i\neq j}^{p-1} \floor{\beta_i}.
$$
It is obvious that the latter is satisfied for sufficiently large $m$. Consequently, \eqref{eq:diagdom}
is true, the submatrix~\eqref{eq:submatice} is non-sigular and therefore using Item 2) of Lemma~\ref{l:omatici}, $\beta_i\in\Q(\delta)$ for $i=1,\dots,p$.
\end{proof}

\begin{lemma}\label{l:deltaAlgebraic2}  
Let $\BB = (\beta_1, \beta_2, \ldots, \beta_p)$ be an alternate base.
Suppose that there exists a constant $\gamma>0$ such that every rational number in $[0,\gamma)$ has eventually periodic $\BB$-expansion. Then 
$\delta=\prod_{i=1}^p\beta_i$ is an algebraic integer. 
\end{lemma}

For the proof of this lemma, we will 
use the following statement which follows from the prime number theorem, see~\cite[p. 494]{HardyWright2008}.

\begin{proposition}[\cite{HardyWright2008}]\label{prop:HardyWright}
Let $\varepsilon >0$. Then there exists $A_1 \in \N$ such that for any $A > A_1$,  the interval $\big(A, A(1+\varepsilon)\big)$ contains at least one rational prime. \end{proposition}

\begin{proof}[Proof of Lemma~\ref{l:deltaAlgebraic2}] 
In the proof of Lemma \ref{l:deltaAlgebraic} we set for $x^{(j)} = \frac{p_j}{q_j} $ any rational numbers from the interval $\frac{1}{\delta^n} I_j$. With this choice we have derived that $\delta$ is a root of a non-zero  integer polynomial $F$ with leading coefficient equal to $-q_1q_2\cdots q_p p_p$. The only condition on the fixed index $n\in \N$ was that $\frac{1}{\delta^n}  < \gamma$. Now we show that a more meticulous choice of $n$ ensures that the interval $\frac{1}{\delta^n} I_j$ contains two fractions $x^{(j)} = \frac{1}{q_j}$  and  $\tilde{x}^{(j)} = \frac{1}{\tilde{q}_j}$, where $q_j$ and $\tilde{q}_j$ are mutually distinct primes. With this choice of two $p$-tuples of rational numbers, we obtain by the same procedure as before two polynomials with root $\delta$, say  $F$  and $\tilde{F}$ from $\Z[X]$, the first one with leading coefficient $-q_1q_2\cdots q_p$, the second one with leading coefficient $-\tilde{q}_1 \tilde{q}_2 \cdots \tilde{q}_p$. Since $I_j$ are mutually disjoint,  $q_1,\dots,q_{p},\tilde{q}_1,\dots,\tilde{q}_{p}$ are distinct primes. Hence, the leading coefficients of $F$ and $\tilde F$ are coprime. By Bézout's lemma, there exists a monic polynomial with integer coefficients with root $\delta$, and thus $\delta$ is an algebraic integer.

\medskip
For the proof of Lemma~\ref{l:deltaAlgebraic2} it is therefore sufficient to demonstrate how to find an integer $n\in \N$ so that each of the intervals $\frac{1}{\delta^n} I_j$, $j=1,\dots,p$, contains two fractions $\frac{1}{q_j}$  and $\frac{1}{\tilde{q}_j}$, such that $q_j$ and $\tilde{q}_j$ are distinct primes. For this, we use Proposition~\ref{prop:HardyWright} stated above.
\medskip 

Denote by $\ell_j$ and $r_j$ the left and right end-points of the interval $I_j$ defined in \eqref{eq: Iintrvaly}, respectively. Since in the lexicographic order the prefix $0^{j}10^{pm-j-1}$ defining the interval $I_{j+1}$ is smaller than the corresponding prefix for $I_{j}$, obviously,  
$\ell_p < r_p < \ell_{p-1} < r_{p-1}< \cdots < \ell_1 < r_1$ and also $\frac{1}{r_1}< \frac{1}{\ell_1}< \cdots < \frac{1}{r_p}< \frac{1}{\ell_p}$.  Set $\varepsilon$ so that 
$$
(1+\varepsilon)^2 = \min \Bigl\{ \frac{r_j}{\ell_j}\ : \  j =1,2,\ldots, p\Bigr\}
$$
To such $\varepsilon$ we find by Proposition~\ref{prop:HardyWright} the number $A_1$.  Now we choose $n\in \N$ such that besides the inequality $\frac{1}{\delta^n}  < \gamma$ we also have  $A_1< \frac{\delta^n}{r_1}$. Each interval $K_j:=\Bigl(\frac{\delta^n}{r_j}, \frac{\delta^n}{\ell_j}\Bigr) $ has its left end-point larger than $A_1$ and the ratio of its right and left end-points satisfies $\frac{r_j}{\ell_j} \geq (1+\varepsilon)^2$.  Proposition~\ref{prop:HardyWright} ensures that $K_j$ contains two distinct primes, say $q_j$ and $\tilde{q}_j$. 
Since $q_j, \tilde{q}_j \in K_j$ we have $\frac{1}{q_j}, \frac{1}{\tilde{q}_j} \in  \bigl(\frac{\ell_j}{\delta^n}, \frac{r_j}{\delta^n}\bigr) = \frac{1}{\delta^n}I_j$ as we wanted to show. 
\end{proof}

\begin{lemma}\label{l:deltaAlgebraic3}  
Let $\BB = (\beta_1, \beta_2, \ldots, \beta_p)$ be an alternate base.
Suppose that there exists a constant $\gamma>0$ such that every rational number in $[0,\gamma)$ has eventually periodic $\BB$-expansion. Then 
$\delta=\prod_{i=1}^p\beta_i$ is a Pisot or a Salem number.
\end{lemma}

\begin{proof}
By Lemma~\ref{l:deltaAlgebraic2}, $\delta$ is an algebraic integer. It remains to show that no algebraic conjugate of $\delta$ is in modulus strictly larger than 1, i.e. $\delta$ is a Pisot or a Salem number. Assume for the contradiction that there exists an algebraic conjugate $\eta$ of $\delta$ distinct from $\delta$, which has modulus $|\eta| >1$. 

Fix now $n\in \N$ so that $\frac{1}{\delta^{n-1}} < \gamma$ and $\delta^n \neq \eta^n$. 
Then for every positive integer $m$ the set    
$$
I(m) = \{ x \in [0,1) : d_{\BB}(x) \ \text{has a prefix } \  0^{pn-1}10^{pm}  \} 
$$ 
is an interval and every $x \in I(m)$ satisfies $x < \frac{1}{\delta^{n-1}} < \gamma$. Choose    $x^{(m)}\in I(m) \cap \mathbb{Q}$.  The $\BB$-expansion of  $x^{(m)}$ is eventually periodic and has the form
$$
d_{\BB}(x^{(m)}) = 0^{pn-1}10^{pm} x_{p(n+m)+1}^{(m)}x_{p(n+m)+2}^{(m)}\cdots.
$$
Hence, 
\begin{equation}\label{eq:alfy1}
x^{(m)} = \frac{1}{\delta^n} + \frac{1}{\delta^{n+m}}\sum_{k=1}^{+\infty}\frac{d^{(m)}_{k}}{\delta^k} , \quad \text{with} \ \ {d^{(m)}_{k}} \in \mathcal{D} \subset \mathbb{Q}(\delta)\,.
\end{equation} 
As the expansion is eventually periodic, we can apply to the latter relation the field isomorphism $\psi$ induced by $\psi(\delta) = \eta$  to obtain 
\begin{equation}\label{eq:alfy2}
x^{(m)} = \frac{1}{\eta^n} + \frac{1}{\eta^{m+n}}\sum_{k=1}^{+\infty}\frac{\psi\bigl(d^{(m)}_{k}\bigr)}{\eta^k} \in  \mathbb{Q}(\eta)\,.
\end{equation}
Since the digit set $\mathcal{D}$ is finite, there exist constants $K$ and $\widetilde{K}$ such that $|d| \leq K$ and $|\psi(d)|\leq \widetilde{K}$ for every $d \in \mathcal{D}$. The infinite sums in \eqref{eq:alfy1} and \eqref{eq:alfy2} can therefore be bounded by  
$$
\Bigl|\  \sum_{k=1}^{+\infty}\frac{d_{k}}{\delta^k}\ \Bigr| \leq \frac{{K}}{|\delta| - 1}\qquad \text{and} \qquad 
\Bigl| \ \sum_{k=1}^{+\infty}\frac{\psi(d_{k})}{\eta^k}\ \Bigr| \leq \frac{\tilde{K}}{|\eta| - 1}\,.
$$
Subtracting \eqref{eq:alfy1} and \eqref{eq:alfy2} with the use of these estimates, we derive
$$
\Bigl|\frac{1}{\delta^n} - \frac{1}{\eta^n}\Bigr| \ \leq  \ \frac{1}{|\delta|^{m+n}} \ \frac{{K}}{|\delta| -1}\  +  \ \frac{1}{|\eta|^{m+n}} \ \frac{\tilde{K}}{|\eta| -1}
$$
for every sufficiently large $m \in \N$.
The right hand side of the inequality tends to $0$ with $m$ increasing to the infinity, whereas the left hand side is a positive number independent of $m$. Hence by contradiction, a conjugate $\eta$ of $\delta$ in modulus greater than 1 cannot exist. 
\end{proof}

\begin{proof}[Proof of Theorem~A]
It suffices to combine statements of Lemmas~\ref{l:deltaAlgebraic} and~\ref{l:deltaAlgebraic3}.
\end{proof}

\section{Proof of Theorem~\ref{thm:necessary}}\label{sec:necessary}

Theorem~B gives a necessary condition for an alternate base $\BB$ so that it satisfied Property (PP). Its proof is an adjustment of arguments used by Akiyama~\cite{A98} for the case of alternate bases.

\begin{lemma}\label{l:deltaAlgebraic4}  
Let $\BB = (\beta_1, \beta_2, \ldots, \beta_p)$ be an alternate base.
Suppose that there exists a constant $\gamma>0$ such that every rational number in $[0,\gamma)$ has purely periodic $\BB$-expansion. Then 
\begin{enumerate}
    \item[1)] $\delta$ is an algebraic unit, $\beta_i\in\Q(\delta)$ for  $i=1,\dots,p$, and 
    \item[2)] for any non-identical embedding $\psi:\Q(\delta)\to\C$, the vector $\big(\psi(\beta_1),\dots,\psi(\beta_p)\big)$ is not positive. 
\end{enumerate}
\end{lemma}

\begin{proof}
Assume that the $\BB$-expansion of every rational number in $[0,\gamma)$ is purely periodic. From Lemma~\ref{l:deltaAlgebraic2}, we know that $\delta=\prod_{i=1}^p \beta_i >1$ is an algebraic integer, say of degree $D$, and thus its norm ${\rm Norm}(\delta)$ is an integer. 
Since  $\beta_i \in \Q(\delta)$, and the alphabet $\mathcal{D} \subset \Q(\delta)$ is finite, there exists a positive integer $q$ such that every digit $d \in \mathcal{D}$ can be written in the form  $d=\frac1q  h^{(d)}(\delta)$, where $h^{(d)} \in \Z[X]$ with ${\rm deg\,}h^{(d)} \leq D-1$.   

In order to show that $\delta$ is an algebraic unit, we need to verify that ${\rm Norm}(\delta)=\pm 1$. Assume the contrary, i.e.\ $\Delta:=|{\rm Norm}(\delta)| \geq 2$.  This allows us to find an 
 $n \in \N$ such that the constant $\gamma>0$ from the definition of Property (PP) satisfies
 $$
 \min\{\gamma, \delta^{-D}\} > \frac{1}{q\,\Delta^n} =:x
 $$
Since  $x \in [0, \gamma)\cap \mathbb{Q}$ and $x < \delta^{-D}$,   the $\BB$-expansion of $x$ is 
$$
d_{\BB}(x)=(0^{pD}x_{pD+1}x_{pD+2}\cdots x_{ps})^\omega
$$  
for some $s \in \N$.  
By \eqref{eq:Xperiodic}, we have
\begin{equation}\label{eq:x}
 x(\delta^s - 1) = \sum_{k=D+1}^s d_k\delta^{s-k} = 
 \sum_{k=D+1}^s \frac1{q}h^{(d_k)}(\delta)\delta^{s-k}\,,
\end{equation}
i.e.\ we have for some $ F \in  \Z[X]$ of degree ${\rm deg}(F)\leq s-2$ that $x(\delta^s - 1)=\frac1qF(\delta)$.
Substituting the value of $x$, the latter implies that 
$\delta^s - 1 =  \Delta^n \,F(\delta)$. In other words,  $\delta$ is a root of the polynomial $X^s - \Delta^n F(X) -1 \in\Z[X]$.  The constant term of the polynomial is equal to $-1 \mod \Delta$. As the norm of a root of any integer polynomial divides its constant term, we derive that  $\Delta$ divides $-1$ and hence $\Delta = |{\rm Norm}(\delta)| = 1$, which is a contradiction, proving Item 1).

We will prove Item 2) by contradiction. Assume that there exists a non-identical embedding $\psi:\Q(\delta)\to\C$ such that $\psi(\beta_i)>0$ for every $i=1,\dots,p$. We therefore have $\psi(d)>0$ for any digit $d\in{\mathcal D}$. 
Applied on the number $\delta$, we obtain $\psi(\delta)=\eta>0$. Since $\delta$ is a Pisot or a Salem number, we have  $\eta\in(0,1)$. 

Applying $\psi$ on~\eqref{eq:x} gives
$$
(\eta^s-1)x = \sum_{k=D+1}^s \psi(d_k)\eta^{s-k}, 
$$
This is a contradiction, since the number on the left hand-side is negative, whereas the right hand-side is $\geq 0$.
\end{proof}

\begin{proof}[Proof of Theorem~B]
    It suffices to combine the statements of Lemma~\ref{l:deltaAlgebraic}, Lemma~\ref{l:deltaAlgebraic3} and Lemma~\ref{l:deltaAlgebraic4}.
\end{proof}

\section{Proof of Theorem~\ref{thm:sufficient}}\label{sec:sufficient}

Theorem~C expresses a sufficient condition for an alternate base $\BB$ to satisfy Property (PP).
Two auxiliary statements are needed. 

\begin{lemma}\label{le:pomoc1} 
Let $\delta>1$ be a Pisot unit and  $\mathcal{D}$ denote the alphabet defined in \eqref{eq:Dabeceda}.  Assume that $\beta_1, \beta_2, \ldots, \beta_p \in \mathbb{Q}(\delta)$.   
Then for every $c>0$ there exists $k\in \N$ with the following property: 

If the $\BB$-expansion of $y$ is of the form $y=\sum_{i=-k}^{-1}d_i\delta^i$ with $d_{-1}, d_{-2}, \ldots, d_{-k} \in \mathcal{D}$ and $d_{-k}\neq 0$, then $|\psi(y)|\geq c$ for at least one non-identical embedding $\psi: \mathbb{Q}(\delta) \to \mathbb{C}$.  
\end{lemma}

\begin{proof} Denote by $D$ the degree of the algebraic number $\delta$. The following facts hold true:

\begin{enumerate}
    \item[Fact 1)] Since $\delta$ is an unit, $\delta^{-i}$ belongs to the ring $\Z[\delta] = \{a_0 +a_1\delta +a_2\delta^2 + \cdots + a_{D-1}\delta^{D-1}: a_j \in \Z\}$ for every $i \in \N$.  

    \item[Fact 2)] Since $\beta_1, \beta_2, \ldots, \beta_p \in \mathbb{Q}(\delta)$, the digit set $\mathcal{D}$ is a finite subset of $\mathbb{Q}(\delta)$, and thus there exists $q \in \N$ such that $\mathcal{D} \subset \frac1q\mathbb{Z}[\delta]$. 

    \item[Fact 3)] The set $S_H:=\{x \in \mathbb{Z}[\delta]: |\psi(x)| <H \text{ \ for every embedding  $\psi:\mathbb{Q}(\delta) \to \mathbb{C}$} \}$ is finite  for any $H>0$.  
\end{enumerate}

We prove the statement of the lemma by contradiction.
Assume that there exists a constant $c > 0$ such that  for any $k \in \N$ we can find a number $y^{(k)}$ with the $\BB$-expansion of the form 
$$
y^{(k)}=\sum_{i=-k}^{-1}d_i^{(k)}\delta^i,\qquad d^{(k)}_{-1},d^{(k)}_{-2}, \ldots,d^{(k)}_{-k} \in \mathcal{D},\  d^{(k)}_{-k}\neq 0,
$$ 
such that $|\psi(y^{(k)})| < c$ for every  non-identical embedding $\psi: \mathbb{Q}(\delta) \to \mathbb{C}$.
Note that the $\BB$-expansion of a number is unique, and thus the sequence $\bigl(y^{(k)}\bigr)_{k\in\N}$ is injective. Necessarily, the set $S:=\{qy^{(k)}: k\in \N\}$ is infinite.

On the other hand, by Facts 1) and  2),  we have  $S\subset \mathbb{Z}[\delta]$.  Moreover, the identical embedding satisfies $|\psi(qy^{(k)})| = |qy^{(k)}| \leq  q\mu \sum_{i\geq 1} \delta^{-i} = q\mu \frac{1}{\delta -1} $, where $\mu$ denotes  $\max\{|d| : d \in \mathcal{D}\}$. 
Setting  $H :=\max \{qc,q\mu \frac{1}{\delta -1}\} $, we have $S\subset S_H$ and by Fact 3), $S$ is finite, which is a contradiction.   
\end{proof}

\begin{lemma}\label{le:pomoc2}  
With the notation of Lemma \ref{le:pomoc1} it further holds:

There exists a constant $\kappa>0$ such that for every $z \in \N_{\BB} \setminus \delta\, \mathbb{N}_{\BB}$, we have $|\psi(z)|\geq \kappa$ for at least one non-identical embedding $\psi : \mathbb{Q}(\delta) \to \mathbb{C}$. 
\end{lemma}

\begin{proof} In this proof we again abbreviate ``embedding of $\mathbb{Q}(\delta)$ into $\mathbb{C}$"  to ``embedding". Denote 
$$
H:=\max\{|\phi(z)|: z \in \mathbb{Z}_{\BB} \text{ and } \phi \text{ a non-identical embedding}  \}.
$$
We will show that  $H$ is finite. For this purpose, 
denote
\begin{equation}\label{eq:FImi}
\begin{aligned}
\Phi &= \max\{|\psi(\delta)|: \psi \text{ a non-identical embedding} \}, \quad \text{and}\\    
\mu &=\max\{|\psi(d)| :d \in \mathcal{D} \text{ and } \phi \text{ a non-identical embedding}\}.
\end{aligned}
\end{equation}
Since $\delta$ is Pisot and  $1\in \mathcal{D}$,  obviously $\Phi \in (0,1)$ and  $\mu \geq 1$. 

For every $z \in \mathbb{N}_{\BB}$ and for every non-identical embedding  $\phi$ we have the following inequalities 
$$
|\phi(z)| = \Bigl|\phi\Bigl(\sum_{i=0}^k d_i\delta^i\Bigr)\Bigr| \leq \sum_{i=0}^k |\phi(d_i)| \cdot |\phi(\delta)|^i \leq \mu \sum_{i=0}^{+\infty}  \Phi^i = \frac{ \mu}{1-\Phi}, 
$$  
Consequently, $H <  \frac{ \mu}{1-\Phi}$.
We now use Lemma~\ref{le:pomoc1} to find the index $k$ for the choice $c= 2H> 0$.  

Let $z \in \mathbb{N}_{\BB} \setminus \delta\, \mathbb{N}_{\BB} $. Then $\frac{z}{\delta^k}$ can be written as $\frac{z}{\delta^k} = w + y$, where $w \in \mathbb{N}_{\BB}$  and $y$ satisfy assumptions of Lemma~\ref{le:pomoc1}. Thus there exists a non-identical embedding $\psi$ such that  $|\psi(y)| \geq 2H$. As $w \in \mathbb{N}_{\BB}$,  from the definition of $H$ we clearly obtain $|\psi(w)|\leq H$. 
Therefore 
$$
|\psi(z)|  = |\psi(y)+ \psi(w)|  \,.|\psi(\delta)|^k\geq   \Bigl(|\psi(y)|- |\psi(w)|\Bigr)\,. |\psi(\delta)|^k\geq H \,.|\psi(\delta)|^k 
$$ 
To conclude the proof, it suffices to set
$\kappa = H \cdot\min\{|\phi(\delta)|^k: \phi \text{ an embedding}  \}$. 
\end{proof}

\begin{remark}\label{re:hodneNul} 
 By~\cite{MPS23-LNCS}, requirement that Property (F) is satisfied forces  $d_{\BB^{(i)}}(1)$ to be finite for every $i =1,2,\ldots, p$. There exists only finitely many factors of the form $a0^\ell b$ with  $a,b \neq 0$, $\ell\in\N$, belonging to the language of the string $d_{\BB^{(i)}}(1)$ for some $i$. Choose $r \in \N$ so that  $\ell < rp$ for every such factor $a0^\ell b$. This choice of $r$ ensures that if both $x_1x_2 \cdots x_{pk} 0^\omega$ and $y_1y_2y_3
\cdots $ are $\BB$-admissible, then also $x_1x_2 \cdots x_{pk} 0^{pr} y_1y_2y_3\cdots $ is $\BB$-admissible. 
\end{remark}

\begin{proof}[Proof of Theorem~C] 
We first show that $\BB=\BB^{(1)}$ satisfies Property (PP).

Let $c>0$ denote the constant from Lemma~\ref{le:pomoc2} and let $r\in \N$ be as chosen in Remark \ref{re:hodneNul}. 
Set $\gamma = \min\{c, \delta^{-r}\}$.

\medskip
Let  $x \in \mathbb{Q}\cap (0,\gamma)$. We will use the fact that there exist infinitely many $N\in \N$  such that  $x(\delta^N - 1) \in \mathbb{Z}[\delta]$, see~\cite{A98}.
Among such exponents, choose $N$ so that  $|\phi(\delta)|\cdot|\phi(\delta)^N-1| <1$, for every  non-identical embedding $\phi$. This is possible, because $\delta$ is a Pisot number, i.e.\ $|\phi(\delta)|<1$.
Denote $z =x(\delta^N - 1)\in \mathbb{Z}[\delta]$.

Since $\BB$ satisfies property (F), the $\BB$-expansion of $z$ is of the form $z=\sum_{j=s}^n d_j\delta^j$, where $s,n \in \mathbb{Z}, s\leq n,  d_s, d_{s+1}, \ldots, d_n \in \mathcal{D}$  and $d_s\neq 0$.

\medskip
We first show using Lemma~\ref{le:pomoc2} that  $z \in \mathbb{N}_{\BB}$, in other words that $s\geq 0$.  
Indeed, as $\delta^{-s}z \in \mathbb{N}_{\BB}\setminus \delta \, \mathbb{N}_{\BB}$,  we can find a  non-identical embedding $\psi$ such that    
$$
\begin{aligned}
c&< |\psi(\delta^{-s}\,z)| = |\psi(\delta)|^{-s} \cdot|x| \cdot|\psi(\delta)^N -1| = \\
&=
|\psi(\delta)|^{-s-1} \,|x| \,|\psi(\delta) |\cdot|\psi(\delta)^N -1| < c  \,|\psi(\delta)|^{-s-1}   
\end{aligned}
$$
Since $|\psi(\delta)| < 1$, the latter inequality implies $-s-1\leq -1$, i.e.\  $s\geq 0$. 
\medskip

Now since $z \in \mathbb{N}_{\BB}$  and $z = x(\delta^N -1) < \delta^{-r} (\delta^N-1) < \delta^{N-r}$, the $\BB$-expansion of $\frac{z}{\delta^N}$ is of the form 
$$
d_{\BB}\Bigl(\frac{z}{\delta^N}\Bigr) = 0^{rp}z_{rp+1}z_{rp+1} \cdots z_{pN}0^\omega \,.
$$
By Remark \ref{re:hodneNul}, the purely periodic 
string $(0^{rp}z_{rp+1}z_{rp+1} \cdots z_{pN})^\omega$ is  $\BB$-admissible.  Moreover, the string represents the number
$$
\frac{z}{\delta^N} + \frac{z}{\delta^{2N}} 
+ \frac{z}{\delta^{3N}} +\cdots =  \frac{z}{\delta^N - 1} =x. 
$$
In order to show that $\BB^{(i)}$ satisfies (PP) for $i\geq 2$, it suffices to recall the result of~\cite{MPS23-LNCS} which states that if $\BB$ satisfies the finiteness property, then so do all of its shifts. 
\end{proof}

\section{A class of bases with (PP)}\label{sec:example}

In this section we present a class of alternate bases in which any rational number in the interval $[0,1)$ has purely periodic expansion.  

\begin{proposition}\label{pro:gama1}
Let $\delta>1$ be the positive root of the polynomial $x^2-(m+1)x-1$, $m\geq1$, and let $\beta_1=\frac{\delta}{\delta-1}$, $\beta_2=\delta-1$.   
Then every rational number in the interval $[0,1)$ has purely periodic expansion, i.e.\ $\gamma(\BB)=1$.
\end{proposition}

It can be easily computed that $\floor{\beta_1}=1$, $\floor{\beta_2}=m$ and
\begin{equation}\label{eq:rozvojejednicekpriklad}
\begin{aligned}
d_{\BB}(1) &= 11,\\ 
d_{\BB^{(2)}}(1)&=m01,    
\end{aligned}\qquad\text{which implies}\qquad
\begin{aligned}
d_{\BB}^*(1) &= (10)^\omega,\\ 
d_{\BB^{(2)}}^*(1)&=m0(01)^\omega.      
\end{aligned}
\end{equation}
Note that for $m=2$, we obtain $\BB=(\frac{5+\sqrt{13}}{6},\frac{1+\sqrt{13}}{2})$, which is a shift of the example well studied in~\cite{CC21}.

\begin{proof}
The number $\delta$ is a quadratic Pisot unit of the class studied in~\cite{S80} for which Schmidt proved that the constant $\gamma$ from the (PP) property is equal to 1.
Thus every rational number $r\in(0,1)$ has a purely periodic expansion in the Rényi numeration system with base $\delta$, i.e.
$d_{\delta}(r)=(d_1d_2\cdots d_n)^\omega$ where, by the Parry condition~\cite{P60}, for each $i\in\N$, 
$$
d_{i+1}d_{i+2}d_{i+3}\cdots \prec_{\text lex} ((m+1)0)^{\omega} = d_{\delta}^*(1).
$$
In particular, the digits $d_i$ belong to $\{0,1,\dots,m+1\}$, and each digit $d_i=(m+1)$ is followed by the digit $d_{i+1}=0$.
We thus have
\begin{equation}\label{eq:r}
r=\big(\frac{d_1}{\delta} + \frac{d_2}{\delta^2}+\cdots +\frac{d_n}{\delta^n}\big)\frac{\delta^n}{\delta^n-1}
=\big(\underbrace{{d_1}{\delta}^{n-1} + {d_2}\delta^{n-2} + \cdots +{d_n}}_{=:z}\big)\frac{1}{\delta^n-1}.  
\end{equation}

We now justify that the last digit in the period is
$d_n=0$. Indeed, 
    from \eqref{eq:r} we derive that 
    $$
    r(\delta^n-1)=\sum_{k=1}^{n-1}d_k\delta^{n-k} + d_n.
    $$
Apply the non-identical automorphism $\psi$ of the field $\Q(\delta)$, which is induced by $\psi(\delta)=\delta'=-\frac1{\delta}$. We obtain
$$
\begin{aligned}  
0 &>r((\psi(\delta))^n-1) = \sum_{k=1}^{n-1}d_k(\psi(\delta))^{n-k} + d_n >\\ 
&>d_n - (m+1)\sum_{i=0}^{+\infty}\frac{1}{\delta^{2i+1}}  = d_n - \frac{m+1}{\delta}\frac{1}{1-\frac1{\delta^2}} = d_n-1.
\end{aligned}
$$
This gives that $d_n<1$, whence $d_n=0$.

Let us now inspect the expansion of the rational number $r$ in base $\BB$. In~\eqref{eq:uprava}, we have shown that a representation of a number $x$ in an alternate base $\BB$ can be viewed as a representation of $x$ in a single base $\delta$ with digits in the alphabet ${\mathcal D}$. In our case, ${\mathcal D}=\{a\beta_2+b:a,b\in\N,a\leq 1, b\leq m\}$.
Replacing the digits $a\beta_2+b\in {\mathcal D}$ by the pair
$\boxed{ab}$, one converts the representation of a number $x$ in base $\delta$ into a representation of $x$ in the alternate base $\BB$. The Rényi expansion~\eqref{eq:r} of $r$ is purely periodic with the period being the number $z={d_1}{\delta}^{n-1} + {d_2}\delta^{n-2} + \cdots +{d_n}$. We thus have a $\BB$-representation of $r$
\begin{equation}\label{eq:rboxed}
    r=\underbrace{\begin{array}{|c|c|c|c|c|} \hline  0d_1&0d_2& \cdots& 0d_{n-1}&00\\ \hline\end{array}}_{z} \
    \underbrace{\begin{array}{|c|c|c|c|c|} \hline  0d_1&0d_2& \cdots& 0d_{n-1}&00\\ \hline\end{array}}_{z} \ \cdots
\end{equation}

Note that this $\BB$-representation of $r$ need not be the $\BB$-expansion of $r$, for: (i) it may contain blocks $\boxed{0(m\!\!+\!\! 1)}$ which do not represent digits of the alphabet ${\mathcal D}$, (ii) the order of blocks $\boxed{ab}$ does not respect the condition of admissibility of $\BB$-expansions (Theorem~\ref{thm:admis}) given by~\eqref{eq:rozvojejednicekpriklad}.

We will define an algorithm which rewrites the representation~\eqref{eq:rboxed} into an admissible $\BB$-expansion, by dealing with problems (i) and (ii). We will use
 simple relations which hold for $\beta_2$ and $\delta$ and can be derived straightforwardly from their definition, namely
\begin{eqnarray}
    m+1 &=&\beta_2 + \frac{\beta_2}{\delta}   \label{eq:A},\\
    m+ \frac{1}{\delta} &=& \beta_2 \label{eq:B}\\
       m+ \frac{\beta_2}{\delta}+ \frac{\beta_2}{\delta^2} &=& \beta_2 + \frac{m}{\delta} \label{eq:C}.
\end{eqnarray}

The algorithm is designed as follows:

\begin{itemize}
\item[(i)] Since $d_i=(m\!+\! 1)$ implies $d_{i+1}=0$, each block $\begin{array}{|c|}
\hline
    0(m\!+\! 1)\!   
\\ \hline
\end{array}$ appears in the sequence~\eqref{eq:rboxed} in the pair $\begin{array}{|c|c|}
\hline
    0(m\!+\! 1)\!&00  
\\ \hline
\end{array}$. Rewrite 
$$
\begin{array}{|c|c|}
\hline
    0(m\!+\! 1)\!&00  
\\ \hline
\end{array}  \mapsto \begin{array}{|c|c|}
\hline
   10 & 10
\\ \hline
\end{array}
$$
The two strings represent the same value, thanks to~\eqref{eq:A}.
Use this rewriting rule for the string $z=\begin{array}{|c|c|c|c|c|} \hline  0d_1&0d_2& \cdots& 0d_{n-1}&00\\ \hline\end{array}$ of~\eqref{eq:rboxed} until we have a representation of $z$ as a string of blocks $\boxed{0b}$, $b\in\{0,\dots,m\}$, and blocks $\boxed{10}\boxed{10}$. Moreover, the  string of blocks representing $z$ ends either in $\boxed{00}$ or in $\boxed{10}\boxed{10}$. We will say that it is reduced.

\item[(ii)]
From the condition of admissibility of $\BB$-expansions (Theorem~\ref{thm:admis}) and the expansions of 1 in the considered base~\eqref{eq:rozvojejednicekpriklad}, we derive that the reduced sequence of blocks obtained in step (i) may contain non-admissible substrings only of the form 
 $\begin{array}{|c|c|} \hline  0m&0b\\ \hline\end{array}$, $1\leq b\leq m$ (say Type A) or $\begin{array}{|c|c|c|} \hline  0m&10&10\\ \hline\end{array}$ (Type B).
 For them, we can use the following rewriting rules:
\begin{equation}\label{eq:rulefortypeA}
\begin{array}{|c|c|} \hline  0m&0b\\ \hline\end{array} \mapsto \begin{array}{|c|c|} \hline  10&00(b\!\!-\!\!1)\\ \hline\end{array}    
\end{equation}
and
\begin{equation}\label{eq:rulefortypeB}
\boxed{0m}\boxed{10}\boxed{10} \mapsto \boxed{10}\boxed{0m}\boxed{00}.
\end{equation}
One can verify using~\eqref{eq:B} and~\eqref{eq:C} that  these rewriting rules preserve the value of the string.

We proceed from left to right by induction on the position of the left-most occurrence of a forbidden substring. In the reduced sequence of blocks we find the left-most occurrence of a forbidden string of Type A or Type B. Using the rewriting rules~\eqref{eq:rulefortypeA} or~\eqref{eq:rulefortypeB}, respectively, we obtain again a reduced sequence of blocks whose left-most occurrence of a forbidden string is smaller by two blocks. 

After finitely many steps, we obtain a reduced sequence of blocks which does not contain any forbidden strings, i.e.\ is admissible in base $\BB$.

At last, realize that step (ii) produces a $\BB$-admissible sequence of blocks representing the number $z$ which ends with the block $\boxed{00}$, or the block $\boxed{10}$. 
In order to obtain the $\BB$-expansion of $r$, we need to concatenate the string for $z$ infinitely many times. The condition on the ending block of $z$ ensures that such a concatenation remains $\BB$-admissible.
\end{itemize}

\end{proof}

\begin{example}
Consider the expansion of the number $r=3/4$ in base $\BB=(\frac{5+\sqrt{13}}{6},\frac{1+\sqrt{13}}{2})$, which corresponds to the above class with $m=2$.
 We first find the R\'enyi $\delta$-expansion of $3/4$ for $\delta=\frac{5+\sqrt{13}}{6}\cdot\frac{1+\sqrt{13}}{2}=\frac12(3+\sqrt{13})$ with minimal polynomial $x^2-3x-1$.

 By the standard greedy algorithm for expansion of a number in one real base $\delta$, we obtain $d_\delta(3/4)=(211230)^\omega$. 
 This gives a $\BB$-representation of $3/4$ written in blocks
$$
\big(\,\boxed{02}\boxed{01}\boxed{01}\boxed{02}\boxed{03}\boxed{00}\,\big)^\omega
$$
From that, by rewriting rule
$$
\begin{aligned}
\boxed{03}\boxed{00} &\mapsto \boxed{10}\boxed{10},\\
\boxed{02}\boxed{01} &\mapsto \boxed{10}\boxed{00},\\
\boxed{02}\boxed{10}\boxed{10} &\mapsto \boxed{10}\boxed{02}\boxed{00}.
\end{aligned}
$$
we will derive the $\BB$-expansion of $3/4$. In the first step, we use the rule $\boxed{03}\boxed{00} \mapsto \boxed{10}\boxed{10}$ to get a $\BB$-representation in the admissible alphabet ${\mathcal D}$. We have
$$
\begin{gathered}
\big(\,\boxed{02}\boxed{01}\boxed{01}\boxed{02}\boxed{03}\boxed{00}\,\big)^\omega\\
\downarrow\\
\big(\,\boxed{02}\boxed{01}\boxed{01}\boxed{02}\boxed{10}\boxed{10}\,\big)^\omega\\    
\end{gathered}
$$
Next, we use the rules $\boxed{02}\boxed{01} \mapsto \boxed{10}\boxed{00}$,
$\boxed{02}\boxed{10}\boxed{10} \mapsto \boxed{10}\boxed{02}\boxed{00}$ to get rid of forbidden strings of digits. We proceed from left to right,
$$
\begin{gathered}
\big(\,\boxed{02}\boxed{01}\boxed{01}\boxed{02}\boxed{10}\boxed{10}\,\big)^\omega\\
\downarrow\\
\big(\,\boxed{10}\boxed{00}\boxed{01}\boxed{02}\boxed{10}\boxed{10}\,\big)^\omega\\    
\downarrow\\
\big(\,\boxed{10}\boxed{00}\boxed{01}\boxed{10}\boxed{02}\boxed{00}\,\big)^\omega\\    
\end{gathered}
$$
Finally, we obtain
$$
d_{\BB}(3/4) = (100001100200)^\omega.
$$
\end{example}

\begin{remark} \label{rem:comments}
Consider the shift $\BB^{(2)}$ of the base $\BB$ of Proposition~\ref{pro:gama1}. In particular, 
let $\delta>1$ be the positive root of the polynomial $x^2-(m+1)x-1$, $m\geq1$ and let $\BB =(\beta_2, \beta_1)$, where $\beta_2=\delta-1$ and $\beta_1=\frac{\delta}{\delta-1}$. By Theorem~\ref{thm:sufficient}, we know that $\BB^{(2)}$ has also Property (PP). Let us make an estimate of the constant $\gamma(\BB^{(2)})$. 
One can easily verify the following: \\
If $m$ is even, i.e.\ $m=2k$, for $k\geq 1$, then 
$$
d_{\BB^{(2)}}(\tfrac12) = k\bigl(000k0(k\!+\!1)\bigr)^\omega.
$$
\medskip
If $m$ is odd, i.e.\ $m=2k+1$, for $k\geq 0$, then  
$$
d_{\BB^{(2)}}(\tfrac12) = k\bigl(0(k\!+\!1)\bigr)^\omega.
$$
In both cases, $\tfrac12$ does not have a purely periodic $\BB^{(2)}$-expansion. Thus $\gamma(\BB^{(2)})\leq \frac12$. Our computer experiments suggest that the bound is not reached. For $m=2$ we estimate $\gamma(\BB^{(2)}) \approx 0.41$. 
\end{remark}



\section{Comments}

Theorem~\ref{thm:SchmidtLepsi} states that every rational has an eventually periodic expansion in base $\BB=(\beta_1,\dots,\beta_p)$ only if the product $\delta=\prod_{i=1}^p\beta_i$ is a Pisot or a Salem number and $\beta_i\in\Q(\delta)$ for every $i$. A partial converse of this statement holds, as shown in~\cite{CCK23}, (cf.~Theorem~\ref{thm:PKreczman}). 
It seems to be a difficult question to decide whether an analogue of Item (2) of Theorem~\ref{thm:PKreczman} can be shown for Salem $\delta$.

Let us mention that the problem of Salem numbers appears already in case that $p=1$, i.e.\ in case of Rényi $\beta$-expansions. The question of periodicity in non-Pisot bases can be circumvented by considering representations in base $\beta$ without requiring that they come out from the greedy algorithm. Such idea was first presented in~\cite{BMPV17}, where the authors show that any base  $\beta$ which is either a rational or an algebraic integer without conjugates on the unit circle admits a finite alphabet ${\mathcal A}$ of digits so that any element in the field $\Q(\beta)$ has a periodic $(\beta,{\mathcal A})$-representation. Vávra in~\cite{Vavra21} then extends this result to all algebraic bases, including Salem numbers.



\bibliographystyle{siam}
\bibliography{references}

\end{document}